\documentclass{amsart}
\usepackage{amssymb}
\usepackage{amsmath}
\usepackage{mathrsfs}
\usepackage{tikz}
\usepackage{braket}
\usepackage{stmaryrd}

\newtheorem{thm}{Theorem}[section]
\newtheorem{prop}[thm]{Proposition}
\newtheorem{lemma}[thm]{Lemma}

\theoremstyle{definition}
\newtheorem{defin}[thm]{Definition}
\newtheorem{que}{Question}

\newcommand\power{\mathop{\mathscr P}}
\newcommand\X{\mathscr X}

\newcommand\sem[1]{\llbracket #1\rrbracket}
\newcommand\dom{\text{dom}}
\newcommand\FV{\mathop\text{FV}}
\newcommand\henkin[2]{ {\displaystyle #1 \atop \displaystyle #2} }
\newcommand\fd[2]{\left[ #1 \mathord\rightarrow #2 \right]}
\newcommand\mvd[2]{\left[ #1 \mathord\twoheadrightarrow #2 \right]}
\newcommand\da[2]{D(#1,#2)}
\renewcommand\H{\mathop{\mathcal H}}
\renewcommand\P{\mathop{\mathcal P}}
\renewcommand\L{\mathop{\mathcal L}}

\newcommand{\Br}{{\mathop{\mathrm{Br}}}}
\renewcommand\models\vDash
\newcommand\nmodels\nvDash
\newcommand\imp\rightarrow
\newcommand\ekv\leftrightarrow
\newcommand\restrictto{\mathord\upharpoonright}
\newcommand\SBL{\mathrm{SBL}}

\title{Generalized quantifiers in Dependence logic}
\author{Fredrik Engstr\"om}
\date{\today}

\address{Department of Philosophy, Linguistics and Theory of Science \\ University of Gothenburg\\ 
 Box 200, 405 30 G\"oteborg, Sweden}
\email{fredrik.engstrom@gu.se}
\thanks{The author was partially supported by the EUROCORE LogICCC LINT 
	program and the Swedish Research Council. \\ The author would like to thank 
	the anonymous referee who in several ways improved this paper.\\ This paper 
	will we published in Journal of Logic, Language and Information. The final 
	publication is available at springerlink.com. DOI: 
	10.1007/s10849-012-9162-4}

\begin{document}

\begin{abstract}
We introduce generalized quantifiers, as defined in Tarskian semantics
by Mostowski and Lindstr\"om, in logics whose semantics is based on teams instead of assignments,
e.g., IF-logic and Dependence logic. Both the monotone and the non-monotone case is considered.

It is argued that 
to handle quantifier scope dependencies of generalized quantifiers in a 
satisfying way the dependence atom in Dependence logic is not well suited and 
that the \emph{multivalued dependence atom} is a better choice.  This atom is 
in fact definably equivalent to the \emph{independence atom} recently 
introduced by V\"a\"an\"anen and Gr\"adel.
\end{abstract}

\keywords{Dependence logic, Independence Friendly logic, Generalized quantifiers, Multivalued dependence}

\maketitle

\section{Introduction}

Dependencies appear in many guises in both formal and natural languages.  
Several logical systems have been constructed bringing such quantifer scope 
dependencies to the forefront of the syntactical construction, but none of 
these handle generalized quantifiers, one of the basic tools in logic, 
descriptive complexity theory, and formal linguistics.  The purpose of this 
paper is to introduce generalized quantifiers in these logical frameworks in 
such a way that branching, i.e. non linearity, of generalized quantifiers can 
be handled naturally in the logic itself.

Dependence logic, proposed by V\"a\"an\"anen \cite{Vaananen:2007}, is an 
elegant way of introducing dependencies between variables into the object 
language. It can also deal with branching of existential and universal 
quantifiers, but so far it cannot handle generalized quantifiers.  In this 
paper we present a way of extending Dependence logic with generalized 
quantifiers.

\subsection{Generalized quantifiers and natural languages}

When giving (parts of written) natural languages, such as English, a formal 
model theoretic semantics, such as in \cite{Montague:1970}, several problems 
naturally surface. One is how
to treat \emph{determiners} such as {\tt all}, {\tt some} and {\tt
most}. It turns out that Mostowski's and Lindstr\"om's (see for example
\cite{Mostowski:1957} and \cite{Lindstrom:1966}) notions of generalized
quantifiers are most useful when formalizing expressions with determiners, see
\cite{Stanley:2006} for a thorough account of
this.

According to Mostowski and Lindstr\"om a quantifier of type $\langle 
n_1,\ldots,n_k\rangle$, where $n_i$ are positive natural numbers,
is a class (in most cases a \emph{proper class}) of
structures in the finite relational signature $\set{R_1,\ldots,R_k}$ where 
$R_i$ is of arity $n_i$, closed under taking isomorphic images.
For example, the meaning of the determiner \emph{most} is commonly the type 
$\langle 1,1\rangle$ quantifier $$\text{\tt most} = \set{(M,A,B) : |A \cap B| 
	\geq |A \setminus B|}.$$
Thus, a possible formalization of the proposition ``most boys are tall'' is
$$\mathtt{most}\, x,y \ (Bx, Ty)$$
where $B$ is the predicate of being a boy and $T$ that one of being
tall.  The truth condition for this proposition is then 
\begin{multline*}
(M,B,T) \models \mathtt{most}\, x,y \ (Bx, Ty) \ \text{ iff } \ (M,B,T) \in
\mathtt{most}\ \text{ iff }  \ |B \cap T| \geq |B \setminus T|,
\end{multline*}

which seems to coincide with the intuitive truth condition for the 
proposition. Note that we are subscribing to the sloppy style of not 
distinguishing between the predicate symbols and the predicates, e.g., in the 
above truth condition $B$ stands for both the predicate symbol in the formula 
and a subset of the domain.

Given a generalized quantifier $Q$ of type $\langle n_1,\ldots,n_k\rangle$ and 
a domain $M$, let the \emph{local} quantifer $Q_M$ be defined as
$$Q_M = \set{\langle A_0,A_1,\ldots,A_k \rangle \subseteq M^{n_1}\times \ldots \times M^{n_k}|  (M,A_0,A_1,\ldots,A_k) \in Q}.$$
Observe that local quantifiers are just sets of relations over the domain $M$, they are not generalized quantifiers in the strict sense. Generalized quantifiers in the strict sense we sometimes call \emph{global} when need is to distinguish them from \emph{local} quantifiers.

\subsection{Dependence and independence in natural languages}

In  \cite{Hintikka:1974} Hintikka claims that the proposition
\begin{quote}
\strut\llap{($\ast$)\ \ }\it Some relative of each villager and some relative of each townsmen hate
each other. 
\end{quote} 
ought to be interpreted as 
$$\henkin{\forall x \exists y}{\forall z \exists w} A(x,y,z,w)$$
where $A(x,y,z,w)$ is the quantifier free formula expressing that if $x$
is a villager and $z$ is a townsman then $y$ is a relative of $x$, $w$ is a
relative of $z$, and $y$ and $w$ hate each other, and 
$\henkin{\forall x \exists y}{\forall z \exists w}$ is the partially
ordered quantifier studied by Henkin in \cite{Henkin:1961}, whose semantics 
is easiest expressed by its skolemization:
$$\exists f,g \forall x,z \ A(x,f(x),z,g(z)),$$
thus $y=f(x)$ may only depend on the value of $x$ and $w=g(z)$ only on $z$.

However this interpretation of ($\ast$) in terms of the branching Henkin quantifier
has been strongly objected to (see for example Barwise \cite{Barwise:1979} and Gierasimczuk and Szymanik \cite{Gierasimczuk:2009}) and other more natural
examples of branching have been given, such as Barwise's example from \cite{Barwise:1979}
\begin{quote}
\strut\llap{($\dagger$)\ \ }\it Most of the dots and most of the stars are all connected by lines. 
\end{quote}
It should be rather clear, we think, that one natural reading of this is that there
is a set of stars $A$ which includes most stars, and a set $B$ of dots
including most dots, such that each star in $A$ is connected to
each dot in $B$. That is the branching reading of the sentence. Branching here means that the choice of the set of stars may not depend on the choice of any particular dot in the earlier chosen set of dots. 

It seems hard to find natural examples in natural languages of
branching involving only the first order quantifiers $\exists$ and $\forall$. Examples involving generalized quantifiers as in ($\dagger$) above is easier to find. 
Another example of when branching reading is natural is with numerical quantifiers as in the following example from Davies \cite{Davies:1989}.
\begin{quote}
\strut\llap{($+$)\ \ }\it
Two examiners marked six scripts.
\end{quote}
Maybe the most natural reading of ($+$) is
$$\henkin{\exists^{=2} x}{\exists^{=6} y}\bigl( E(x) \land S(y) \land M(x,y)\bigr),$$
where $E$ is the predicate of being an examiner, $S$ that of being a script, and $M(x,y)$ the relation of $x$ marking $y$.
The numerical quantifiers $\exists^{=k}$ are, even though definable in first order logic,  proper generalized quantifiers.

To be able to handle branching readings of sentences like ($\ast$) in a coherent logical framework Hintikka developed
Independence Friendly logic, or IF-logic for short, in which statements of 
the form ``there exists $x$, chosen independently of $\bar y$, such that'' 
can be expressed by the formal construction
$$\exists x / \bar y \ A(x,\bar y).$$
Here $\bar y$ is a finite sequence of variable $y_0,y_1,\ldots,y_{n-1}$.
We say that $\exists x / \bar y$ is a \emph{slashed} quantifier.
However IF-logic, as it stands, cannot handle generalized quantifiers, the 
chief
example of branching in natural languages. This paper introduces generalized quantifiers in
IF-logic, and many of its variants such as Dependence Friendly logic (DF-logic) and Dependence logic.

Barwise (see \cite{Barwise:1979}), among others, argues that for monotone\footnote{A quantifier $Q$ is monotone if given $A \subseteq B \subseteq M$ such that $(M,A) \in Q$ then $(M,B) \in Q$.} quantifiers $Q_1$ and $Q_2$ of type
$\langle 1 \rangle$ the branching of $Q_1$ and
$Q_2$
$$\henkin{Q_1x}{Q_2y}A(x,y)$$
should be interpreted as
$$\Br(Q_1,Q_2)xy \ A(x,y),$$
where $\Br(Q_1,Q_2)$ is the type $\langle 2 \rangle$ quantifier
$$\set{(M,R) | \exists A \in Q_1, B \in Q_2, A \times B \subseteq
R}.$$
We take this as the definition of branching of two monotone quantifiers.
The correctness of that definition seems to be rather universally agreed upon.
Thus, our definition of quantifiers in DF-logic should reflect upon this.

It could be worth noting that for \emph{monotone} quantifiers $Q_1$ and $Q_2$ a formula $Q_1x\,Q_2y\, \varphi$ can be translated into existiential second-order logic with $Q_1$ and $Q_2$ used as second-order predicates in the following way:
$$
\exists X \bigl( Q_1(X) \land \forall x \mathord\in X \exists Y \bigl(Q_2(Y) \land \forall y \mathord\in Y \varphi\bigr)\bigr).
$$
In this formula it is clear that the \emph{second-order} variable $Y$ depends on the \emph{first-order} variable $x$. By moving the $\exists Y$ outside of the scope of $\forall x \in X$ we can break this dependence. The resulting formula then becomes:
$$
\exists X \exists Y \bigl( Q_1(X) \land Q_2(Y) \land \forall x \mathord\in X \forall y \mathord\in Y \varphi\bigr)\bigr),
$$
which is equivalent to the branching reading:
$
\Br(Q_1,Q_2)xy \ \varphi,
$
giving some evidence on the correctness of the definition of $\Br(Q_1,Q_2)$.

In the next section we will define both IF-logic and DF-logic, but first take a look at another variant of IF-logic 
developed by V\"a\"an\"anen \cite{Vaananen:2007} called \emph{Dependence logic}.

\subsection{Dependence logic and related logics}

The syntax of Dependence logic is  that of first order logic
together with new atoms, the dependence atoms. There is one
dependence atom for each arity  written
$\fd{t_1,\ldots,t_n}{t_{n+1}}$.\footnote{When V\"a\"an\"anen introduced Dependence logic he used the notation  $\mathord=(t_1,\ldots,,t_n,t_{n+1})$ for $\fd{t_1,\ldots,t_n}{t_{n+1}}$, however we prefer the latter notation.} For
simplicity we will assume that all formulas are written in negation
normal form, i.e., all negation signs occuring in a formula occur in
front of an atomic formula. This is to make some technicalities easier, the 
downside of this approach is that negation cannot be treated in a 
compositional way. More on this later. Note also that negation in Dependence 
logic is not contradictory negation as; for example, we will see later that
$\nmodels \forall x,y (\fd{x}{y} \lor \lnot \fd{x}{y})$.

To define a compositional semantics for Dependence logic we need to consider
\emph{sets of assignments} called \emph{teams}. Formally, an assignment
is a function $s: V \to M$ where $V$ is a finite set of variables and
$M$ is the domain under discussion. A team (on the domain $M$) is a set of assignments of some fixed finite set of variables $V$, i.e., a subset of $\set{s| s : V \to M}$ for some finite set of variables $V$. If $V=\emptyset$ there is only one assignment $V \to M$, the empty assignment, denoted by $\epsilon$. Please observe that the team of the empty assignment $\set{\epsilon}$ is different from the empty team.

Given an assignment $s: V \to M$ and $a \in M$ let $s[a/x]: V \cup \set{x} \to M$ be the assignment:
$$
s[a/x]: y \mapsto  
\begin{cases}
s(y)  &\text{ if $y \in V \setminus \set{x}$, and}\\
a  &\text{ if $x=y$.}
\end{cases}
$$

The domain of a (non-empty) team
$\dom(X)$ is the set of variables $V$.
The condition $M ,X \models  \varphi$ means that the formula $\varphi$ of 
Dependence logic
is satisfied in the structure $M$ by the \emph{team} $X$. We use the notation $M,s \models\varphi$ for ordinary Tarskian satisfaction of the first order formula $\varphi$ under the assignment $s$. 
We call this type of semantics where a formula is satisfied by a team, not just a single assignment, Hodges semantics\footnote{Hodges in \cite{Hodges:1997} invented this framework in order to give IF-logic a compositional semantics.} to distinguish it from ordinary Tarskian semantics.

The truth
conditions for $M ,X \models  \varphi$ are the following:

\begin{align*}
 M ,X \models  R(\bar t) &\text{ iff } \forall s \in X: M,s \models R(\bar t) \\
 M ,X \models  \lnot R(\bar t) &\text{ iff } \forall s \in X : M,s
\models \lnot R(\bar t) \\
M ,X \models  \fd{t_1,\ldots,t_n}{t_{n+1}} &\text{ iff }
\forall s,s' \in X 
\\ &\hskip 1cm \bigwedge_{1\leq i \leq n}
t_i^{M,s}
=t_i^{M,s'}  \rightarrow  t_{n+1}^{M,s} =t_{n+1}^{M,s'}\\
M ,X\models \lnot \fd{t_1,\ldots,t_n}{t_{n+1}} &\text{ iff } X =
\emptyset \\
M ,X \models  \varphi \land \psi &\text{ iff }M ,X \models 
\varphi \text{ and } M ,X \models  \psi \\
M ,X \models  \varphi \lor \psi &\text{ iff } \exists Y \cup Z = X: M
,Y \models  \varphi \text{ and } M,Z \models \psi \\
M ,X \models  \exists y \varphi &\text{ iff } \exists f: X \to M, \text{ such that } M,{X[f/y]}
\models \varphi \\
M ,X \models  \forall y \varphi &\text{ iff }  M,X[M/y]
\models \varphi.
\end{align*}
Here $t^{M,s}$ is the interpretation of the term $t$ in the model $M$ under the assignment $s$, $$X[M/y] \text{ is the team } \set{s[a/y] | s \in X, a \in M}$$ of
assignments, and when ever $f: X \to M$, $X[f/y] \text{ is }\set{s[f(s)/y] | s \in X}.$
Observe that for some teams $X$ we have $M,X \nmodels \fd{x}{y}$ and $M,X 
\nmodels \lnot \fd{x}{y}$. In fact this is the case when $M$ has at least two 
elements and $X$ is the full team of all assignments of $x$ and $y$. 
Therefore, $\nmodels \forall x,y (\fd{x}{y} \lor \lnot \fd{x}{y})$.
This illustrates that negation is not  contradictory negation.

The free variables of a formula is defined in a recursively way, like in first 
order logic, with the extra base case of the dependence atom: all the 
variables in $\bar x$ and $y$ are free in the formula $\fd{\bar x}{y}$. Let 
$\FV(\varphi)$ be the set of free variables of  $\varphi$. A sentence is a 
formula without free variables.
We define $M \models \sigma$ for a sentence $\sigma$ to hold if $M , {\set{\epsilon}} \models \sigma$.

By just staring at the definition of satisfaction we can make some remarks. 
First, every formula is satisfied by the empty team, which has as a 
consequence that for any atomic formula $\varphi$ we have both $M,\emptyset 
\models \varphi$ and $M,\emptyset \models \lnot \varphi$. Second, satisfaction 
is preserved under taking subteams:

\begin{prop}
If $M ,X \models  \varphi$ and $Y \subseteq X$ then $M ,Y \models  \varphi$. 
\end{prop}

The next proposition might seem a bit ad hoc at first sight, but its role will later be apparent. It 
 tells us that the truth condition for the existential quantifier is equivalent to the truth condition we later introduce for generalized quantifiers.

\begin{prop}
$M ,X \models  \exists x\varphi$ iff there exists $F: X \to \exists_M$
such that $M,{X[F/x]}  \models\varphi$, where $X[F/x]$ is the team
$\set{s[a/x]| s\in X, a\in F(s)}$.
\end{prop}

Recall that $\exists_M$ is the local existential quantifier, i.e., the set of non-empty predicates on $M$: $\set{A \subseteq M | A \neq \emptyset}$.

Naturally, the semantic value of a formula in Dependence logic is
 the set of teams satisfying the formula. 
\begin{defin}
The semantic value $\sem{\varphi}_M$ of a formula $\varphi$ in the model $M$ is the set of teams satisfying it:
$$\sem{\varphi}_M = \set{ X | \dom(X) = \FV(\varphi) \text{ and } M ,X \models  \varphi}.$$
\end{defin}

Here we have chosen one of two possible paths, the other one would be
to define the semantic value of a formula to be the pair of the set of teams
satisfying the formula and the set of teams that satisfy the negation of the
formula: $\langle \sem{\varphi}_M, \sem{ \varphi^\lnot}_M\rangle$, where $\varphi^\lnot$ is the formula in negated normal form that corresponds to $\lnot \varphi$.  That would have had the advantage of making negation compositional (i.e., a
function of semantic values). However, it would also make the theory technically much more involved. 

It should also be pointed out that Kontinen and V\"a\"an\"anen in
\cite{Kontinen:2009} proved that if
$\varphi$ and $\psi$ are formulas in Dependence logic with the same free variables such
that $\sem{\varphi}_M \cap \sem{\psi}_M = \set{\emptyset}$ then there is a
formula $\sigma$ in Dependence logic
 such that $\sem{\sigma}_M=\sem{\varphi}_M$ and
$\sem{\sigma^\lnot}_M=\sem{\psi}_M$. Thus the ``positive'' and the
``negative'' semantic values, taken to be $\sem{\varphi}_M$ and $\sem{ \varphi^\lnot}_M$ respectively, of formulas are independent, in the sense that only knowing the positive (negative) semantic value of a formula does not give any information on the negative (positive) semantic value of the same formula.

DF-logic has a different syntax than Dependence logic but a similar semantics. Instead of
introducing dependence atoms we introduce new quantifiers\footnote{Observe that these quantifiers are not \emph{generalized} quantifiers in the sense of Lindstr\"om and Mostowski since they are defined using Hodges semantics, not Tarskian semantics.}
$\exists x \backslash \bar y$ where $\bar y$ is a finite sequence of variables. We call $\exists x \backslash \bar y$ a \emph{backslashed} quantifier. $\exists x \backslash \bar y \ \varphi$ has the same truth condition as
$$\exists x (\fd{\bar y}{x} \land \varphi).$$

Independence friendly logic, IF-logic, is syntactically similar to DF-logic but with slashed
quantifiers instead of backslashed ones.\footnote{In fact, what we describe here is, strictly speaking, what Hodges in \cite{Hodges:2008} calls {slash logic} and not IF-logic.}
 There is a non-compositional translation of IF-logic into Dependence logic: Given a sentence
$\sigma$ in IF-logic we replace each occurrence of 
$\exists x / \bar y \ \varphi$ by
$$\exists x (\fd{\bar z}{x} \land \varphi)$$
where $\bar z$ are the variables occurring in $\sigma$ but not in $\bar y$.

These three logics, IF-, DF- and Depedence logic, are all equivalent in the sense that for each
formula in one of the logics there are formulas in the other logics
satisfied by the same teams in the same structures.
The translations from DF-logic to Dependence logic and back are compositional, but the translations to and from IF-logic is not.

IF-logic has one rather strange property which Dependence logic and DF-logic
does not. In IF-logic an extra variable could be used for ``signaling'' as in the following example:
 $$M \not\models \forall x \exists y \mathord/ x \  x=y$$ if
$|M| > 1$, but $$\models \forall x \exists z \exists y \mathord/ x
\ x=y.$$  
Thus quantifying over variables not occurring in a sentence might change
the truth value of that sentence. This is rather counterintuitive, which should give us a slight preference for DF-logic and Dependence logic over IF-logic.

\section{Generalized quantifiers}

We will now give a rather long argumentation leading up to Definition \ref{qdef} which gives truth conditions for generalized quantifiers in logics whose semantics are given in the framework of teams, such as Dependence logic. As will be apparent later, if a generalized quantifier is definable in existential second order logic, ESO, the result of adding the quantifier to Dependence logic will not change the strength of the logic, it will still be of the same strength as ESO.  However, the translation into ESO will not be compositional, see the discussion in Section \ref{conclusion}. The main reason for introducing generalized quantifiers in this framework is not to gain strength, but to give a compositional explanation of branching. 

In the following we fix a structure and let $M$ ambiguously denote it and its domain.
We will ambiguously  use $Q$ to denote both a global quantifier and the
local version on $M$, which really should be denoted by $Q_M$. We write $\sem{\varphi}$ as a shorthand for $\sem{\varphi}_M$. 

Teams are sets of assignments, and thus not
relations, however if $X$ is a team with
$\dom(X)=\set{x_1,\ldots,x_k}$  let
$$X(x_1,\ldots,x_k)=\set{\langle s(x_1),\ldots, s(x_k)\rangle  | s \in X}$$ be the relation on $M$ we get by applying the assignments in $X$ to the tuple $\langle x_1,\ldots,x_k\rangle$.
Furthermore, if $R \subseteq M^k$ let $[R/x_1,\ldots, x_k]$ be the team
$$\set{ \set{\langle x_1,a_1\rangle, \ldots, \langle x_k, a_k\rangle} |
\langle a_1,\ldots, a_k\rangle \in R}.$$
We will be quite sloppy in distinguishing between teams and relations,
instead identifying the team $X$ with the relation $X(\bar x)$ where
$\bar x$ is $\dom(X)$ listed with the indices in increasing order, and $R$ with
$[R/x_0,\dots,x_{k-1}]$ where $k$ is the arity of $R$.

In \cite{Abramsky:2009} Abramsky and V\"a\"an\"anen give an argument for the correctness of the truth conditions for $\forall$ and $\exists$ in Hodges semantics. In short the argument goes as follows: First they show that Hodges semantics is a special case of a more general construction, that of the free commutative quantale. Second, they show
that the truth conditions of the quantifiers in Hodges semantics are the image
under this general construction of the usual Tarskian truth conditions. Let us see how this works.

Start off by letting the Hodges space be
$$ \H(M^n)= \L(\P(M^n))$$
where $\L(X)$ is the set of order ideals,\footnote{Order ideals are sets closed downwards, i.e., $I \subseteq \P(M^n)$ is an order ideal if for every $A \subseteq M^n$ and any $B \in I$ such that $B \subseteq A$ $B \in I$.} or \emph{down sets}, of the ordered set $X$ and
$\P(M^n)$ is the power set of $M^n$ ordered by set  inclusion.
Given a formula with $n$ free variables in Dependence logic the set of relations corresponding to the teams satisfying the formula is an element of $\H(M^n)$, we therefore think of $\H(M^n)$ as the set of possible semantic values of formulas.\footnote{Observe that not all elements of $\H(M^n)$ are semantic values of formulas in Dependence logic, see \cite{Kontinen:2009} for a complete characterization of elements of $\H(M^n)$ which are.}
Since the elements of $\H(M^n)$ are all closed downwards we restrict ourselves, at the moment, 
to logics where satisfaction is closed under taking subteams.  
Note that $\emptyset$ is a down set and thus an element of $\H(M^n)$. 

If we reformulate the truth conditions for $\exists$ and $\forall$ in algebraic terms as operations mapping semantic values in $\H(M^{n+1})$ to semantics values in $\H(M^n)$ we get the following.
The Hodges quantifiers $\exists_{\H}$ and $\forall_{\H}$ are families of functions
\begin{align*}
\forall_{\H},\exists_{\H} &: \H(M^{n+1}) \to \H(M^n),\\
\exists_{\H}(\X) &= \set{R| \exists f: R \to M \text{ s.t. } R[f] \in
\X},\\
\forall_{\H}(\X) &= \set{R| R[M] \in \X},
\end{align*}
 where $R[f]= \set{\langle \bar a, f(\bar a) \rangle | \bar a \in
R}$ and $R[M] = \set{\langle \bar a, b \rangle | \bar a \in
R, b \in M}$. 

The truth condition for the existential quantifier can now be restated as:
$$ \sem{\exists x \varphi } = \bigl[\exists_{\H} \bigl(\sem{\varphi}[\bar
y,x]\bigr)/\bar y\bigr],$$ where $\bar y$ are the free variables of $\exists x
\varphi$. The corresponding equality is of course true also for the universal quantifier. 
Let us now see that these truth conditions are \emph{forced} upon us by the operation $\L$. 

Given a function $h: \P(A) \to \P(B)$ we define the Hodges lift of that function as:
$$\L(h): \H(A) \to \H(B), \ \X \mapsto \mathord\downarrow \set{h(X) | X \in \X},$$
where $\mathord\downarrow \X$ is the downward closure of $\X$, i.e.,
$$\mathord\downarrow \X = \set{X | \exists Y \in \X, X \subseteq Y}.$$
To every generalized quantifier $Q$ of type $\langle 1\rangle$ there is a corresponding
function on the Tarskian semantic values: 
$$h_Q : \P(M^{n+1}) \to \P(M^n), \ R \mapsto \set{ \bar a | R_{\bar a} \in Q},$$
where $R_{\bar a} = \set{b | \langle \bar a,b\rangle \in R}$.
Now the truth condition for $\exists$ and $\forall$ in the Hodges setting is just the image under $\L$ of the truth conditions for $\exists$ and $\forall$ in the Tarskian setting, in the sense that: $\exists_{\H} = \L(h_\exists)$,  and $\forall_{\H}  = \L(h_\forall)$.
These facts follow easily from the definitions, but see Proposition \ref{prop:4}.

We do not have to stop here. Let us see what happens if we start with some other quantifier $Q$ of type $\langle 1\rangle$ and argue in the same way that led us to the truth conditions for $\exists$ and $\forall$ in the Hodges setting.
Thus, for a generalized quantifier $Q$, let us write  $Q_{\H}$, or $\L(Q)$, for $\L(h_Q)$.
Let $Y[F] = \set{\langle\bar a,b\rangle | \bar a \in Y, b \in F(\bar a)}$. 

\begin{lemma}
Suppose $\X \subseteq \power(M^k)$ and $Q$ a monadic quantifier. (a) If $Q$ is such that $\emptyset \notin Q$, then
$$\set{h_Q(X) |X \in \X} =  \set{ Y | \exists F : Y \to Q \text{ s.t. } Y[F] \in \X}.$$
(b) Furthermore, for any $Q$, if $\X$ is a down set then
$$\set{\set{\bar a | X_{\bar a} \in Q  } |X \in \X}$$
is also a down set.
\end{lemma}
\begin{proof}
(a) Follows from the fact that $Y[F]_{\bar a} = F(\bar a)$ 
if $\bar a \in Y$ and $Y[F]_{\bar a} = \emptyset$ otherwise. 

(b) Is immediate.
\end{proof}

\begin{prop}\label{prop:3} For any $Q$  of type $\langle 1\rangle$ and any $\X \in \H(M^k)$ we have
$$Q_{\H}(\X) =  \set{ Y | \exists F : Y \to Q \text{ s.t. } Y[F] \in
\X}.$$ 
\end{prop}
\begin{proof}
Follows directly from the lemma whenever $\emptyset \notin Q$.
On the other hand if $\emptyset \in Q$ then $Q_{\H}(\X) = \P(M^n)$
whenever $\X \neq \emptyset$ and $Q_{\H}(\emptyset) = \emptyset$. Also
if $F$ is the constant function $F(\bar a) = \emptyset$ then
$Y[F]=\emptyset$, thus $\set{ Y | \exists F : Y \to Q \text{ s.t. }
Y[F] \in \X}$ is $\emptyset$ or $\P(M^n)$ depending on whatever $\X$ is
the empty set or not.
\end{proof}

This all leads up to the following truth condition:
\begin{defin}\label{qdef}
Let $Q$ be a monotone generalized quantifiers $Q$ of type $\langle 1 \rangle$ and $\varphi$ some formula in a logic whose semantics is based on teams. We define what it means for the team $X$ to satisfy the formula $\varphi$ by the following truth condition.
$$M ,X \models  Qx\varphi \text{ iff there exists } F: X \to Q \text{ such that } M,{X[F/x]} \models \varphi. $$
\end{defin}

This applies even for non-monotone quantifiers but for those quantifiers $Q$ the truth condition above does not make a whole lot of sense as the following example shows.
Let $M=\mathbb N$ and $Q=\set{A}$ where $A$ is the set of even numbers. According to the truth condition above $M,{\set{\epsilon}} \models Qx (x=x)$ since there is a team $X=A(x)$ such that $M
,X \models  x=x$.

For this reason let us, for now, restrict the definitions to monotone quantifiers. 

We do not however need to restrict to type $\langle 1
\rangle$ as the definition easily can be extended to all
quantifiers of  type $\langle k\rangle$:
$$M ,X \models  Q\bar x \varphi \text{ iff there exists } F: X \to Q \text{ such that } M,{X[F/\bar x]}  \models\varphi.$$
Here $X[F/\bar x]$ is the team $\set{s[a_1/x_1,\ldots a_k/x_k] | s \in X, \langle a_1,\ldots a_k\rangle \in F(s)}$.

The following easy proposition suggests that we indeed have the right truth condition, at least for monotone quantifiers of type $\langle 1 \rangle$. Given some language $L$ let $L(Q)$ the set of first order formulas in that language extended with the generalized quantifier $Q$.

\begin{prop}\label{prop:4} Below, let $Q$ be a monotone quantifier of type $\langle 1 \rangle$.
\begin{enumerate}
\item $\L(\exists)(\X)=\set{Y | \exists f : Y \to M \text{ s.t. } Y[f] \in \X}$.
\item $\L(\forall)(\X)=\set{Y | Y[M] \in \X}$.
\item For L($Q$)-formulas $\varphi$ and teams $X$, $M ,X \models  \varphi$ iff
for all $s \in X$, $M ,s \models  \varphi$.
\item For L($Q$)-sentences $\sigma$, $M,\epsilon \models \sigma$ iff $M ,{\set{\epsilon}}\models \sigma$.\footnote{Observe that $M ,\epsilon\models \sigma$ uses the ordinary Tarskian truth conditions, but $M, {\set{\epsilon}} \models \sigma$ Hodges semantics.}
\item $\L(Q_1Q_2)=\L(Q_1) \circ \L(Q_2)$, where $Q_1Q_2$ is the iteration (product) of $Q_1$ and $Q_2$ and $\L(Q_1) \circ \L(Q_2)$ is just the ordinary composition of functions (observe that this equality is really an infinite conjunction of equalities since $\L(Q)$ is a family of functions).
\end{enumerate}
\end{prop}
\begin{proof}
(1) The right-to-left inclusion is immediate using Proposition \ref{prop:3}. The other inclusion follows from the fact that if $Y[F]
\in \X$ and $f: Y \to M$ is such that $f(\bar a) \in F(\bar a)$ for
each $\bar a \in Y$ then $Y[f] \subseteq Y[F] \in \X$ and thus $Y[f]
\in \X$.

(2) Any function $F: X \to \forall$ has to be the constant function
taking $s$ to $M$. Thus $\L(\forall)(\X)=\set{Y | Y[M] \in \X}$ follows from Proposition \ref{prop:3}.

(3) The argument is an induction on the formula $\varphi$, the only non trivial case being when $\varphi$ is $Qx\psi$.
If $M ,X \models  Qx\psi$ then there is a function $F: X \to Q$ such
that $M, X[F/x]\models \psi$, which means that for each $s \in X$
there is a set $F(s) \in Q$ such that $F(s) \subseteq \psi^{M,s}$,
where $\psi^{M,s}=\set{a \in M | M,{s[a/x]} \models \psi}$.
By the monotonicity of $Q$ we have $\psi^{M,s} \in Q$ and thus that
$M ,s \models  Qx\psi$ for every $s \in X$. For the other direction
suppose $M ,s \models  Qx\psi$ for every $s \in X$, and let $F(s)$ be
$\psi^{M,s}$.

(4) Follows directly from (3).

(5) By unwinding the definitions we get
\begin{multline*}
\L(Q_1) \circ \L(Q_2) (\X)= \L(Q_1)(\set{Y | \exists F : Y \to Q_2 \text{ s.t. }
Y[F] \in \X}) \\ = \set{Z | \exists G : Z \to Q_1 \text{ s.t. } \exists F:Z[G] \to
Q_2,
Z[G][F] \in \X} \\ = \set{Z | \exists H : Z \to Q_1Q_2 \text{ s.t. } Z[H] \in \X}.
\end{multline*}
The left to right inclusion in the last equality comes from the fact that if such $F$ and $G$ exist
then we can define $H(s)$ to be $$\set{\langle a,b\rangle | a \in G(s),
b \in F(s,a)} \in Q_1Q_2$$ and then $Z[H] = Z[G][F]$. 
For the other inclusion assume that such an $H$ is given and let $a
\in G(s)$ if $H(s)_a \in Q_2$ and $F(s,a)=H(s)_a$, then $F: Z \to
Q_1$, $G:Z[F] \to Q_2$ and $Z[G][F] \subseteq Z[H] \in \X$. 
\end{proof}

\subsection{Quantifiers and dependence}

Proposition \ref{prop:4} states that the truth condition for generalized 
quantifiers in the Hodges setting behaves nicely when applied to formulas 
without dependence atoms.
The reason for introducing generalized quantifiers in Hodges semantics however 
is to use them with dependencies. Let us see how well they handle relations of 
dependences and independences.

First let us try to see what happens if we introduce the dependence atom of Dependence logic into our logic.
If $Q$ contains no singletons and not the empty set then $$M \nmodels  Qx (\fd{}{x}\land x=x)$$ as long as $X$ is non-empty.\footnote{$\fd{}{x}$ is a short-hand for $\fd{\emptyset}{x}$, i.e., the statement that $x$ is \emph{constant}.} This is counterintuitive since the sentence $Qx (\fd{}{x} \land x=x)$ should be equivalent to $Qx \ x=x$. 
There are also problems with the notion of definability as the next example shows

Assume that $Q$ is definable by a first order sentence $\sigma$ with $P$ as
the only unary predicate. This means that 
$$M \models Qx\varphi \text{ iff } M \models
\sigma[\varphi/P]$$ for all first order formulas $\varphi$ such that no free variables of $\varphi$ occur in $\sigma$.\footnote{$\sigma[\varphi/P]$ is $\sigma$ with all occurences of $P(x_1,\ldots,x_k)$ replaced by $\varphi(x_1,\ldots,x_k)$.}
It would be natural to think that $\sigma$ also defines $Q$ in
Dependence logic.
However if $Q=\exists^{\geq 2}$, $\sigma$ is
\[\exists x \exists y (x \neq y \land Px \land Py),\]
and $\varphi$ is $\fd{}{z}$, then $\not\models \exists^{\geq 2} z \varphi$. But
for any  $|M|> 1$, $M \models \sigma[\varphi/P]$. Thus, introducing dependence atoms into the language seems to destroy nice properties of the logic. However, we think that the dependence atom should take the blame for this, and not the truth condition for the generalized quantifier.

There are two different solutions for handling dependencies in the setting with generalized quantifiers. Either one redefines the dependence atom, or one skips dependence as an atomic property altogether and define slashed and/or backslashed versions of the generalized quantifiers, very much as is done in DF- and IF-logic. Let us start with the latter suggestion and postpone the definition of a new dependence atom until section \ref{depatom}.

\begin{multline*}
M ,X \models  Qx \backslash \bar y\, \varphi \text{ iff there exists } 
F: X \to Q \text{ such that } \\M ,X[F/x] \models  \varphi 
\text{ and $F$ depends only on the values of } \bar y. 
\end{multline*}
Or slightly more formally:
\begin{defin}
$M ,X \models  Qx \backslash \bar y \varphi$ iff there exists 
\[G: X\restrictto \bar y \to Q \text{ such that } M ,X[G/x] \models  \varphi,
\]
where $s\restrictto \bar y=\set{\langle v,a\rangle | \langle v,a\rangle \in s, v \in
\bar y}$, $X\restrictto \bar y=\set{s\restrictto \bar y : s \in X}$ and $X[G/x]=X[F/x]$ where $F(s)=G(s\restrictto \bar y)$.
\end{defin}

We also define $M ,X \models  Qx / \bar y \,\varphi $ in the obvious way: iff there exists $$F: X\restrictto(\dom(X) \setminus \bar y) \to Q$$ such that $M ,X[F/x] \models  \varphi$.

Let us denote first order logic with the generalized quantifiers $Q$ by $\text{SBL}(Q)$ when we allow both slashed and backslashed versions of the quantifiers $\exists$, $\forall$, and $Q$. 
Note that we can translate backslashed quantifiers into formulas where we only allow slashed ones, and vice versa. However, these translations are not compositional and therefore we include both slashed and backslashed quantifiers in the logic $\text{SBL}(Q)$.

\begin{prop}\label{prop:subteams}
Let $Q$ be a monotone quantifier of type $\langle k \rangle$.
$\mathrm{SBL}(Q)$ is closed under taking subteams, i.e., if $M ,X \models  \varphi$ and $Y \subseteq X$ then $M ,Y \models  \varphi$ for all formulas $\varphi$ in $\mathrm{SBL(Q)}$.
\end{prop}
\begin{proof}
Easily seen by checking the truth conditions of the slashed and the backslashed quantifiers. 
\end{proof}

Observe that 
\begin{gather*}
\exists F: X \to \exists_M \text{ s.t. } X[F/x] \in \X 
\text{ and $F$ only depends on } \bar y \\ \text{ iff }\\ 
\exists f: X \to M \text{ s.t. }  X[f/x] \in \X 
\text{ and $f$ only depends on  } \bar y,
\end{gather*}
for every down set $\X$, making this new definition of $\exists x \backslash \bar y$ compatible with the old one. 
Also, it is easy to see that both conditions are equivalent to the corresponding sentence in Dependence logic:

\begin{prop} Let $Q$ be a monotone quantifier of type $\langle k \rangle$ and $\varphi$ a formula in $\mathrm{SBL}(Q)$ then
$$M ,X \models  \exists x (\fd{\bar y}{x} \land \varphi)\ \text{ iff }\ M ,X \models  \exists x
\backslash \bar y \ \varphi.$$
\end{prop}

Let $Q_1$ and $Q_2$ be monotone quantifiers of type $\langle k\rangle$ and $\langle l\rangle$ respectively, then
$$\mathop\mathrm{Br}(Q_1,Q_2)=\set{(M,R) | R \subseteq M^{k+l},  \exists A \mathord\in Q_1\, \exists B \mathord\in Q_2: A
\times B \subseteq R}.$$ 
The next proposition states that $ Q_1\bar xQ_2\bar y/\bar x $ has the intended
meaning $ \mathop\mathrm{Br}(Q_1,Q_2)$

\begin{prop} If $Q_1$ and $Q_2$ are monotone quantifiers of type $\langle k\rangle$ and $\langle l\rangle$ respectively, then 
$$M,X \models \mathop\mathrm{Br}(Q_1,Q_2) \bar x\bar y \,\varphi\ \text{ iff }\ 
M,X \models Q_1\bar xQ_2\bar y/\bar x \, \varphi,$$
where $\varphi$ is in $\text{SBL}(Q_1,Q_2)$.
\end{prop}
\begin{proof} To simplify notation we only prove this in the case that $k=l=1$.
Assume that $M,X \models \mathop\mathrm{Br}(Q_1,Q_2) xy \,\varphi$, i.e., there if $H: X \to \Br(Q_1,Q_2)$ such that
$M,X[H] \models \varphi$. For $s \in X$ define $F(s)=A$ and $G(s)=B$ where $A \in Q_1$ and $B \in Q_2$ are such that $A \times B \subseteq H(s)$. Now $X[F][G] \subseteq X[H]$ and since $\text{SBL}(Q_1,Q_2)$ is closed under taking subteams we have that
$X[F][G] \models \varphi$. The functions $F$ and $G$ witness that $M,X \models Q_1xQ_2y/x \, \varphi$.

On the other hand if there are such $F$ and $G$ witnessing that$M,X \models Q_1xQ_2y/x \, \varphi$  let $H(s) = F(s) \times G(s)$. Then $X[H]=X[F][G]$ and so $H$ witnesses that 
$M,X \models \mathop\mathrm{Br}(Q_1,Q_2) xy \,\varphi$.
\end{proof}

Thus, the slashed and backslashed quantifiers seem to have the intended meanings. One oddity arises with the universal quantifier:
In Dependence logic we have that $M \nmodels \forall x (\fd{}{x} \land \varphi)$ for every structure $M$ with at least two elements. However with the backslashed universal quantifier we have that $$M \models \forall x \backslash \epsilon \, \varphi \ekv \forall x \varphi.$$
Thus, the constructions $\exists x \backslash \epsilon$ and  $\exists x 
(\fd{}{x} \land \cdot)$ are equivalent, however the analogous constructions 
$\forall x \backslash \epsilon$ and  $\forall x (\fd{}{x} \land \cdot)$ are 
\emph{not} equivalent.

\subsection{Non-monotone quantifiers}

The truth condition for monotone quantifiers cannot be extended to 
non-monotone quantifiers as we have previously seen. However, with a slight 
twist the truth condition can actually be extended. 

In the Tarskian setting we say that $M \models Qx\varphi$ iff $\sem{\varphi} 
\in Q$. When $Q$ is monotone (increasing) this is equivalent to demanding that 
there is some set $A \subseteq \sem{\varphi}$ such that $A \in Q$, a fact we 
used for the truth conditions in the Hodges setting. When dealing with 
non-monotone quantifiers $Q$ we need, apart from that there is a set $A 
\subseteq \sem{\varphi}$ such that $A \in Q$, also that it is the 
\emph{largest} set satisfying $A \subseteq \sem{\varphi}$.

Translating this into the Hodges setting we need to say not only that 
$M,X[F/x] \models \varphi$ but also that  $F$ is a \emph{maximal} function for 
which this holds. However, there might not be a maximal $F$ such that $X[F/x]$ 
satisfies $\varphi$. Instead we demand that there should be one $F$ such that 
$X[F/x]$ satisfies  $\varphi$ and such that every larger $F$ such that 
$X[F/x]$ also satisfies $\varphi$ is mapping $X$ into $Q$. Let us try to 
formalize this in the following definition:

\begin{defin}\label{def:nonmono}
Given $F,F' : X \to \power(M)$ let $F \leq F'$ if for every $s \in X$: $F(s) \subseteq F'(s)$.
Let $Q$ be a type $\langle 1 \rangle$ quantifier. Then $M,X \models Qx\varphi$ 
if there is $F: X \to \power(M)$ such that \begin {enumerate}
\item $M,X[F/x] \models \varphi$ and 
\item for each $F' \geq F$ if $M,X[F'/x] \models \varphi$ then for all $s \in X$: $F(s) \in Q$.
\end{enumerate}
\end{defin}

We call the second condition on $F$ the \emph{largeness condition} since it 
forces the function $F$ to take large sets as values.

This definition generalizes to other types of quantifiers as in the case of 
monotone quantifiers. If $Q$ is of type   $\langle k \rangle$ then  $M,X 
\models Q\bar x \varphi$ iff there is $F: X \to \power(M^k)$ satisfying 
(slight variants of) the two conditions above.

This largeness condition we have added is quite similar to Sher's maximality 
principle for the branching of non-monotone quantifiers, see \cite{Sher:1990}. 
However, as we will see, the above condition gives rise to a slightly stronger 
notion of branching than the one proposed by Sher.

We can clearly add this largeness condition to slashed and backslashed 
non-monotone quantifiers:  $M,X \models Qx/\bar y \,\varphi$ if there is a 
witness $F:X \to \power(M)$ to $M,X \models Qx\varphi$ where $F$ is 
independent of the values of $\bar y$. Observe here that demanding $F$ to be a 
witness for the statement  $M,X \models Qx\, \varphi$ means that $F$ is large 
with respect to \emph{all} functions $F': X \to \power(M)$ and not only those 
that are determined by the values of $\bar y$.  We define  $M,X \models 
Qx\backslash \bar y \, \varphi$ in a similar manner.

\begin{prop}
For $Q$ monotone the  truth condition with the largeness condition in Definition \ref{def:nonmono} is equivalent to the old condition without the largeness condition.
\end{prop}
\begin{proof}
Obvious since if there is $F: X \to Q$ and $F'\geq F$ then $F'(s) \in Q$ for all $s \in X$ by the monotonicity of $Q$.
\end{proof}

\begin{prop}
If $\varphi$ is an $L(Q)$-formula then 
$$ M,X \models \varphi \text{ iff  for all } s \in X: M,s \models \varphi.$$
\end{prop}
\begin{proof}

	The proof is by induction. We only need to check the induction step for the 
	quantifier $Q$. Assume $M,X \models Qx\varphi$. Let $F$ be such that $X[F/x] 
	\models \varphi$ and satisfying the largeness condition, and let $s \in X$.  
	By the induction hypothesis we have that for all $a \in F(s)$: $M,s[a/x] 
	\models \varphi$. Therefore $F(s) \subseteq \sem{\varphi}_{M,s}$ and by the 
	largeness condition on $F$ we know that $\sem{\varphi}_{M,s} \in Q_M$ and 
	therefore $M,s \models Qx\varphi$.

On the other hand if for all $s \in X: M,s \models \varphi$, then we let $F(s) 
= \sem{\varphi}_{M,s}$. It is clear that $F(s) \in Q_M$ for all $s \in X$ and 
also that there cannot be any $F' > F$ such that $M,X[F'/x] \models \varphi$ 
since that would violate the definition of $F$. Therefore $M,X \models 
Qx\varphi$.
\end{proof}

\begin{defin}[Sher \cite{Sher:1990}]
A cartesian product $A \times B$ is \emph{maximal in $R$} if $A \times B \subseteq R$, no $A' \supsetneq A$ satisfies $A'\times B \subseteq R$ and no $B' \supsetneq B$ satisfies $A\times B' \subseteq R$. 

Let the branching of two type $\langle 1\rangle$ quantifiers $Q_1$ and $Q_2$, $\Br^S(Q_1,Q_2)$ be the type $\langle 2\rangle$ quantifier
$$\set{(M,R) | R \subseteq M^2, \exists A\mathord\in Q_1 \exists B \mathord \in Q_2: A \times B \text{ is maximal in } R}. $$
\end{defin}

\begin{lemma}
If $R \notin \Br^S(Q_1,Q_2)$ then for every $A \in Q_1$ and $B \in Q_2$ if $A \times B \subseteq R$  there is either  
\begin{enumerate} 
\item $A' \supseteq A$ such that $A' \notin Q_1$ and $A' \times B \subseteq R$, or
\item $B' \supseteq B$ such that $B' \notin Q_2$ and $A \times B' \subseteq R$.
\end{enumerate}
\end{lemma}
\begin{proof}
	Suppose not and let $A_0$ be the union of all $A' \supseteq A$ such that $A' 
	\times B \subseteq R$. Then $A_0 \times B \subseteq R$, and by the 
	assumption $A_0 \in Q_1$.  Let $B_0$ be the union of all $B' \supseteq B$ 
	such that $A_0 \times B' \subseteq R$. Then $A_0 \times B_0 \subseteq R$ and 
	thus $A \times B_0 \subseteq R$ and by the assumption $B_0 \in Q_2$.

$A_0 \times B_0$ is maximal in $R$ by construction, hence $R \in \Br^S(Q_1,Q_2)$ contradicting the assumption.
\end{proof}

\begin{prop}\label{prop:sher}
Suppose that $\varphi$ is such that $\sem{\varphi}$ is closed downwards ($\varphi$ could for example be a formula of $L(Q)$ or of Dependence logic).
If $M,X \models Q_1x Q_2y / x \varphi$ then $M,X \models \Br^S(Q_1,Q_2)xy \varphi$. 
\end{prop}
\begin{proof}
Assume that $M,X \models Q_1x Q_2y / x \varphi$, i.e., that there are $F,G: X 
\to \power(M)$ satisfying the relevant largeness condition. To prove that $M,X 
\models \Br^S(Q_1,Q_2)xy \varphi$ we need to find an $H: X \to \power(M^2)$ 
witnessing the truth condition. Let $H(s)=F(s) \times G(s)$. We need to prove 
that (1) $M,X[H/xy] \models \varphi$ and (2) that for any $H' \geq H$ such 
that  $M,X[H'/xy] \models \varphi$ we have that $H'(s) \in \Br^S(Q_1,Q_2)$.

(1) Since $X[H/xy]=X[F/x][G/y]$;  $M,X[H/xy] \models \varphi$ follows from the assumption that $$M,X[F/x][G/y] \models \varphi.$$

(2) Assume that $H' \geq H$, $M,X[H'/xy] \models \varphi$ and $H'(s_0) \notin \Br^S(Q_1,Q_2)$ for some $s_0 \in X$. By the lemma we either have $A\notin Q_1$ such that  
$$F(s_0) \times G(s_0) \subseteq A \times G(s_0) \subseteq H'(s_0)$$ or $B \notin Q_2$ such that 
$$F(s_0) \times G(s_0) \subseteq F(s_0) \times B \subseteq H'(s_0).$$
First assume that we have such a $B$. Then $G$ does not satisfy the largeness condition, i.e., that for every $G' \geq G$ if $M,X[F/x][G'/y] \models \varphi$ then $G'(s) \in Q_2$ for all $s \in X$, this is because we could define $G'$ as $G$ except that $G'(s_0)=B$.

On the other hand suppose we have such an $A$ and let $F'$ be as $F$ except that $F'(s)=A$. Then $F' \geq F$ and if we prove that $M,X[F'/x] \models Q_2 y/x\, \varphi$ that contradicts the largeness of $F$ since $F'(s_0) \notin Q_1$. Since $X[F'/x][G/y] \subseteq X[H'/xy]$ and $M,X[H'/xy] \models \varphi$ we have that $M,X[F'/x][G/y] \models \varphi$. We also need to prove that $G$ satisfies the largeness condition. Let $G' \geq G$ be such that $M,X[F'/x][G'/y] \models \varphi$, then $M,[F/x][G'/y] \models \varphi$ and since $G$ satisfies the largeness condition for $X[F/x]$ we know that $G'(s) \in Q_2$ for all $s \in X$, proving the largeness condition for $G$ with $X[F'/x]$.
\end{proof}

We leave the question as wether the proposition holds for general formulas 
$\varphi$ open.

\begin{que}
Is Proposition \ref{prop:sher} also true for formulas $\varphi$ such that 
$\sem{\varphi}$ is not a down set?
\end{que}

The implication in the other direction is in general false as the following example shows. Let 
$$R=\set{\langle 0,0\rangle} \cup \bigl(\set{0,1}\times \set{1,2}\bigr)$$ where $0,1,2 \in M$, see Figure \ref{fig:one}.
Then $$(M,R) \models \Br^S(\exists^{=1},\exists) xy \, R(x,y)$$ since $\set{0} \times \set{0,1,2}$ is maximal in $R$ and $\set{0} $ is in  $\exists^{=1}$ and $\set{0,1,2}$ is in $\exists$. 
However, $$(M,R) \nmodels \exists^{=1}x \exists y/x \, R(x,y)$$ since to get $(M,R),A\times B \models R(x,y)$, $A \in \exists^{=1}$ and $B \neq \emptyset$, we are forced to choose $A=\set{0}$ or $A=\set{1}$. But then the largeness condition for $A$ is not satisfied.

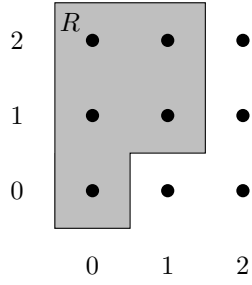
\begin{figure}
\begin{center}
\begin{tikzpicture}[inner sep=0pt]
\fill [lightgray] (.5,.5) rectangle (1.5,1.5);
\fill [lightgray] (.5,1.5) rectangle (2.5,3.5);
\node at (.7,3.25) {$R$};
\foreach \x in {0,1,2} {
           \node at (\x+1,0) {\x};
            \foreach \y in {1,2,3} {
                \node at (\x+1,\y) [circle,fill=black,minimum size=5pt] {};
}
}
\foreach \y in {0,1,2} {
           \node at (0,\y+1) {\y};}
\draw (.5,.5) -- (1.5,.5) -- (1.5,1.5) -- (2.5,1.5) -- (2.5,3.5) -- (.5,3.5) 
-- cycle;
\end{tikzpicture}
\end{center}
\caption{Example of a relation $R$ satisfying $\Br^S(\exists^{=1},\exists)xy$ 
	and $\exists y\exists^{=1} x/y$ but not $\exists^{=1}x\exists y/x$.  
}\label{fig:one}
\end{figure}

This example also shows that in general the two quantifier prefixes $Q_1x 
Q_2y/x$ and $Q_2y Q_1x/y$ are not equivalent:
	$$ (M,R) \nmodels \exists^{=1}x \exists y/x \, R(x,y)$$
but
$$(M,R) \models \exists y \exists^{=1} x/y \, R(x,y).$$

We end this section with two open questions regarding this pathology.

\begin{que}
	Are there natural conditions under which the prefixes $Q_1x Q_2y/x$ and 
	$Q_2y Q_1x/y$ are equivalent?
\end{que}

\begin{que}
	Are there other truth conditions for non-monotone quantifiers (and slashed 
	versions) such that \begin{enumerate}
		\item for monotone quantifiers the truth conditions coincide with the ones 
			for the monotone case, \item for formulas of $L(Q)$ we have $M,X \models 
			\varphi$ iff for all $s \in X$: $M,s\models \varphi$, and
		\item the prefixes  $Q_1x Q_2y/x$ and $Q_2y Q_1x/y$ are equivalent?
	\end{enumerate}
\end{que}

Our proposed truth conditions satisfy (1) and (2), but not (3).

\section{Dependence as an atom}\label{depatom}

Let us now investigate the possibility of defining a new
dependence atom $\da{\bar x}{y}$ giving the intended meaning $Qy \backslash 
\bar x \, \varphi$ to expressions of the form
$Qy (\da{\bar x}{ y} \land \varphi)$.
First we observe that there is no way of doing this if we want to
keep the property of the logic being closed under taking subteams. The
following argument shows this.

Assume that $\da{x}{y}$ is an atom closed under subteams satisfying that
$$\models \forall x \exists^{\geq 3} y \bigl(\da{x}{y} \land R(x,y)\bigr)\ \leftrightarrow\ 
 \forall x \exists^{\geq 3} y\backslash \epsilon \  R(x,y).$$
Fix $M = \set{0,1,2}$, then $$(M,M^2) \models  \forall x \exists^{\geq 3} 
y\backslash \epsilon \  R(x,y).$$ Thus, $X=[M^2/x,y]$ has to  satisfy $\da{x}{y}$. 
By the downward closure of $D$,  we have that the team  
\begin{equation}\label{team}
X=[S/x,y],\text{ where }S=(\set{0,1} \times 
\set{0,1}) \cup (\set{2} \times \set{1,2}),
\end{equation}
satisfies the atom $D$
and thus that $$(M,S) \models \forall x \exists^{\geq 2} y (\da{x}{y} \land R(x,y)),$$ however 
$$(M,S) \not\models   \forall x \exists^{\geq 2} y \backslash
\epsilon \  R(x,y).$$
 This argument shows that no atom $\da{x}{y}$
closed under taking subteams can have the intended effect on both
the quantifiers $\exists^{\geq 2}$ and $\exists^{\geq 3}$.
 
However, by abandoning the property that truth is closed under subteams we can 
define an atom satisfying the right equivalences. To get the mind on the right 
track let us go back and take a look at the formula $Q_1x\,Q_2y\,\varphi$ and 
its translation into existential second order logic with $Q_1$ and $Q_2$ as 
second order predicates:
$$
\exists X \bigl( Q_1(X) \land \forall x \mathord\in X \exists Y \bigl(Q_2(Y) \land \forall y \mathord\in Y \varphi\bigr)\bigr).
$$
In this translation it is clear that the variable $Y$ depends on the variable $x$. Thus a quantifier prefix like $Q_1x\,Q_2y$ gives rise to a dependence in which the value of $x$ determines the \emph{set} $Y$ of possible values for $y$. The new dependence atom, which we denote by 
$\mvd{x_1\ldots,x_k}{x_{k+1}}$, tries to capture this type of dependence in which 			
the \emph{set of possible values} of $x_{k+1}$ is determined by the
values of the variables $x_1,\ldots,x_k$. 
We formalize this idea, but first we need a little bit of notation to work with.

Let $X_s^{\bar y}$, for $s \subseteq s' \in X$ and $\bar y \in \dom(X)$, be the set of possible values of $\bar y$ given $s$, in other words:

\begin{defin}
Given a team $X$, variables $\bar y \in \dom(X)$ and $s \subseteq s'$ for some $s' \in X$, let
$$ X_s^{\bar y} = \set{s'(\bar y) | \exists s' \in X, s \subseteq s' }.$$
\end{defin}

As an example let $X$ be as in \eqref{team} and $s: x \mapsto 1$ then $X^y_s= \set{0,1}$ and if $s': x \mapsto 2$ then $X^y_{s'}=\set{1,2}$.

\begin{defin}
Assume $\bar x,\bar y \in \dom(X)$, then
\begin{itemize}
\item $M ,X \models  \mvd{\bar x}{\bar y}$ iff  for all 
$s\in X$, $X^{\bar y}_{s \restrictto \bar x}= X^{\bar y}_{s \restrictto \bar x\bar z}$, where $\bar z=\dom(X) \setminus \set{\bar x,\bar y}$. 
\item $M ,X \models  \lnot \mvd{\bar x}{\bar y}$ iff $X = \emptyset$.
\end{itemize}
\end{defin}

In fact, this is functional dependence for set-valued functions: Assume that we want to check whether $M ,X \models  \mvd{\bar x}{\bar y}$ or not.
Let $F$ map $s \in X \restrictto \bar x,\bar z$ to the set of possible values of $\bar y$, $X^{\bar y}_s$. Then $M,X\models \mvd{\bar x}{\bar y}$ iff $F(s)$ is determined by the values $s(\bar x)$.

Next we give an equivalent definition for $ \mvd{\bar x}{\bar y}$ this time as a first order property of $X$.
\begin{prop}\label{truthcondition}
 $M ,X \models  \mvd{\bar x}{\bar y}$ iff 
$$
\forall s,s' \mathord\in X  \Bigl( s(\bar x)= s'(\bar x)  \imp  \exists s_0 \mathord\in X 
\bigl(s_0(\bar x,\bar y)=s(\bar x,\bar y) \land s_0(\bar z) = s'(\bar z) \bigr)
\Bigr),
$$
where $\set{\bar z}=\dom(X) \setminus \set{\bar x,\bar y}$. 
\end{prop}
\begin{proof}
Assume $M ,X \models  \mvd{\bar x}{\bar y}$ and $s,s'\in X$ such that $s(\bar x)=s'(\bar x)$. 
Then $X^{\bar y}_{s\restrictto \bar x \bar z} =  X^{\bar y}_{s'\restrictto \bar x \bar z}$.
Clearly $s(\bar y) \in X^{\bar y}_{s\restrictto \bar x \bar z} $ and thus $s(\bar y) \in X^{\bar y}_{s'\restrictto \bar x \bar z} $ which means that there is $s_0 \in X$ such that  $s_0 \supseteq s'\restrictto \bar x \bar z$ and $s_0(\bar y) = s(\bar y)$, i.e., that 
$s_0(\bar x,\bar z)=s'(\bar x,\bar z)$ and $s_0(\bar y)=s(\bar y)$.

For the other implication let $s \in X$, we show that  $X^{\bar y}_{s \restrictto \bar x}= X^{\bar y}_{s \restrictto \bar x,\bar z}$.
It should be clear that  $X^{\bar y}_{s \restrictto \bar x} \supseteq X^{\bar y}_{s \restrictto \bar x,\bar z}$, so let $\bar a \in X^{\bar y}_{s \restrictto \bar x} $. Then there is $s' \in X$ such that $s' \supseteq s \restrictto \bar x$ and $s'(\bar y) = \bar a$. 
By assumption there is a $s_0 \in X$ such that $s_0(\bar x,\bar y)=s'(\bar x, \bar y)$ and $s_0(\bar z)=s(\bar z)$, or in other words $s_0 \supseteq s \restrictto \bar x,\bar z$ and $s_0(\bar y)=\bar a$, i.e., that $\bar a \in  X^{\bar y}_{s \restrictto \bar x,\bar z}$.
\end{proof}

The dependence relation $\mvd{\bar x}{\bar y}$ is what database theorists call \emph{multivalued dependence}, see \cite{Beeri:1977}.

By some easy calculations we find that any $X$ satisfying both $\mvd{\bar x}{y}$ and $\mvd{\bar x,y}{z}$ also satisfies $\mvd{\bar x}{y,z}$, i.e.,
$$\mvd{\bar x}{y}, \mvd{\bar x,y}{z} \models \mvd{\bar x}{y,z}.\footnote{Here, by $\Gamma \models \varphi$ we mean that for every model $M$ and team $X$,  whose domain includes at least all free variable of $\Gamma$ and $\varphi$, if $M,X \models \gamma$ for all $\gamma \in \Gamma$ then also $M,X \models \varphi$.}$$
However it is not in general the case that an $X$ satisfying  $\mvd{\bar x}{y,z}$ satisfies $\mvd{\bar x}{y}$, cf., the case of 
functional dependence where $\fd{\bar x}{y}\land \fd{\bar x}{z} $ is equivalent to $\fd{\bar x}{y,z}$. 
Here by $M ,X \models  \fd{\bar x}{\bar y}$ we mean 
$$\forall s,s' \in X \bigr(s(\bar x)=s'(\bar x) \imp s(\bar y)=s'(\bar y)\bigl).$$

It should also be noted that in the case of functional dependence the dependence atom is \emph{not dependent on context} in the sense that 
$$M ,X \models  \fd{\bar x}{\bar y} \text{ iff } M ,{X\restrictto\bar x,\bar y} \models \fd{\bar x}{\bar y},$$
where $X\restrictto\bar x$ is the team $\set{s \restrictto \bar x | s \in X}$. However multivalued dependencies are dependent on context as the following easy examples shows.
$M ,{X\restrictto x} \models \mvd{}{x}$ is always true disregarding what $X$ is. On the other hand if $s(x)=s(y)=0$ and $s'(x)=s'(y)=1$ then $M, {\set{s,s'}}\nmodels \mvd{}{x}$.

There is a close connection between lossless decomposition of data\-bases and multivalued dependencies: Let $X \bowtie Y$ be the \emph{natural join} of the teams $X$ and $Y$, i.e.,
$$
X\bowtie Y = \set{s: \dom(X) \cup \dom (Y) \to M\ | s\restrictto \dom(X) \in X \text{ and }  s\restrictto\dom(Y)  \in Y}.
$$

\begin{prop}[\cite{Fagin:1977}]\label{fagin77}
$$X \models \mvd{\bar x}{\bar y}\text{ iff } X= (X \restrictto \bar x\bar y) \bowtie (X \restrictto \bar x \bar z),$$ 
where $\bar z$ is $\dom(X) \setminus \set{\bar x,\bar y}$.
\end{prop}

Observe that it follows that  $M ,X \models  \mvd{}{\bar y}$ iff there are teams $Y$ and $Z$ such that
$$\dom(Y)=\set{\bar y}, \dom(Z)=\dom(X) \setminus \set{\bar x,\bar y}, \text{ and } X=Y \bowtie Z.$$ In this case, when $\dom(Y)$ and $\dom(Z)$ are disjoint, the natural join of $Y$ and $Z$ is nothing more than the cartesian product.

Next we prove that the functional dependence may be replaced be multivalued dependence in a certain well-behaved syntactical fragment of Dependence logic. This fragment is as expressive as full Dependence logic at the level of sentences.

\begin{prop}
Let $Q$ be monotone of type $\langle 1 \rangle$ and $\sigma$ a sentence in $\mathrm{SBL}(Q)$ with no slashed quantifiers, 
then the resulting sentence $\sigma^\twoheadrightarrow$ in which all occurrences of 
$Q y \backslash \bar x\ \varphi$ are replaced by
$Q y (\mvd{\bar x}{y} \land \varphi)$ is equivalent to $\sigma$.
\end{prop}
\begin{proof}
We prove the more general statement that for every formula $\psi$ of  $\mathrm{SBL}(Q)$ with no slashed quantifiers, 
the resulting formula $\psi^\twoheadrightarrow$ in which all occurrences of 
$Q y \backslash \bar x\ \varphi$ are replaced by
$Q y (\mvd{\bar x}{y} \land \varphi)$ is equivalent to $\psi$. The proof is by induction on the formula $\psi$. The only non trivial case is when $\psi$ is $Qy
\backslash \bar x \ \varphi$. Then $M ,X \models  \psi$
iff there is $F: X \to Q$ such that $M,X[F/y] \models \varphi$ and
$F(s)$ is determined by the values $s(\bar x)$.
\begin{multline*}
M ,X \models  Qy (\mvd{\bar x}{y} \land \varphi^\twoheadrightarrow)  \text{ iff  } \exists F:
X \to Q\text{ s.t. } M ,X[F/y] \models  \varphi^\twoheadrightarrow, \\ 
\text{ and } M
,X[F/y] \models  \mvd{\bar x}{ y}.  
\end{multline*}
Now $M ,X[F/y] \models  \mvd{\bar x}{ y}$ iff $X[F/y]^y_{s\restrictto \bar x} = X[F/y]^y_{s\restrictto \bar x\bar z}$ for every $s \in X$.
 However 
$$X[F/y]^y_{s\restrictto \bar x\bar z}=F(s),$$ so the result follows from the induction hypothesis.
\end{proof}

It should be clear that if for all $s\neq s' \in X$ there is  $x \in \dom(X) \setminus \set{y}$ 
such that 
$s(x) \neq s'(x)$, i.e., $X(\bar x,y)$ is (the graph of) a partial function $M^k \to M$, then 
$$M ,X \models  \fd{\bar x}{y}
 \text{ iff } M ,X \models  \mvd{\bar x}{y}.$$
Thus, if $y$ is existentially quantified in a sentence of Dependence logic $\sigma$ then
the resulting team $X$ can be assumed to have this property and thus $\fd{\bar x}y$ and
$\mvd{\bar x}{y}$ are interchangeable in the following restricted way:

\begin{defin}
A Dependence logic formula $\varphi$ is \emph{normal} if $\fd{\bar x}{y}$ only occurs as 
$\exists y (\fd{\bar x}{y} \land \psi)$.
\end{defin}

\begin{prop}
If $\varphi$ is normal and $\varphi'$ is the result of replacing atoms $\fd{\bar x}{y}$ by $\mvd{\bar x}{y}$ in $\varphi$, 
then for every $M$ and $X$ 
$$ M ,X \models  \varphi\  \text{ iff } \  M ,X \models  \varphi'.$$ 
\end{prop}
\begin{proof}
Easy induction. 
\end{proof}

This means that under restricted use of the dependence atom we can use
either $\fd{}{}$ or $\mvd{}{}$. Since every sentence of Dependence logic can be expressed by a sentence in DF-logic and those in turn can be expressed by a normal sentence of Dependence logic, we know that the fragment of normal sentences is as strong as full Dependence logic.

Let us call Dependence logic in which $\mvd{}{}$ is used instead of $\fd{}{}$ for \emph{Multivalued Dependence logic} or MVDL for short.

The truth definition of $M ,X \models  \mvd{\bar x}{y}$ is
first order in $X$, see Proposition \ref{truthcondition}, and thus for every formula $\varphi$ in MVDL there
is a sentence $\sigma(R)$  in ESO such that 
$$ M ,X \models  \varphi \text{ iff } (M,X) \models \sigma(R).$$
That means that MVDL is at most as strong as existential second order logic (when it comes to
sentences) and thus as Dependence logic. Also, by translating sentences of Dependence logic into normal sentences and then replacing the functional dependence atom with the multivalued dependence atom we get a sentence of MVDL which is equivalent to the original Dependence logic sentence. Thus, MVDL, Dependence logic and ESO are all of the same strength on the level of sentences.

Recently Galliani proved that MVDL is exactly as strong as existential second order logic also on the level of formulas:
\begin{prop}[\cite{Galliani:2011}]
Let $\X$ be a set of teams on a model $M$, then the following are equivalent:
\begin{itemize}
\item There is a formula $\varphi$ of MVDL such that $\X = \sem{\varphi}^M$.
\item There is a sentence of existential second order logic, ESO, $\sigma$ such that $X \in \X$ iff
$(M,X(\bar x)) \models \sigma$.
\end{itemize}
\end{prop}

Remember that $X(\bar x)$ is the relation corresponding to the team $X$. 

\subsection{Multivalued dependence, independence and completeness}\label{completeness}

In a recent paper by  Gr\"adel and V\"a\"an\"anen  \cite{Gradel:2011}  \emph{independence atoms} are introduced:
\begin{multline*} M ,X \models  \bar y \perp_{\bar x} \bar z \text{ iff } \\
\forall s,s' \mathord\in X \Bigl( s(\bar x)= s'(\bar x)  \imp 
 \exists s_0 \mathord\in X\bigl(s_0(\bar x,\bar y)=s(\bar x,\bar y ) \land s_0(\bar z) = s'(\bar z) \bigr)
\Bigr).
\end{multline*}
This atom also applies to terms $\bar t \perp_{\bar s} \bar t'$ by a slight change of the definition. 

As easily seen, we have
$$
M ,X \models  \mvd{\bar x}{\bar y} \text{ iff } M ,X \models  \bar y \perp_{\bar x} \bar z
$$
where $\bar z = \dom(X) \setminus \set{\bar x,\bar y}$. The logic we get when adding independence atoms to first order logic is called \emph{Independence logic}.

The independence relation introduced by Gr\"adel and V\"a\"an\"anen is in the database theory community known as the \emph{embedded multivalued depencency}. It is usually denoted by $\mvd{\bar x}{\bar y | \bar z}$.

Let us use the notation $D \models \varphi$, where $D$ is a (finite) set of dependence atoms (functional, multivalued or embedded multivalued) and $\varphi$ is a single dependence atom (of the same kind) to mean that any team $X$ (over any domain) satisfying all the dependencies in $D$ also satisfies $\varphi$. It is well known that functional dependence is axiomatizable:

\begin{prop}[\cite{armstrong:1974}]
If $D \cup \set{\varphi}$ is a finite set of functional dependence atoms then
$D \models \varphi$ iff $\varphi$ is derivable from $D$ with the following inference rules:
\begin{itemize}
\item Reflexivity: If $\bar y \subseteq \bar x$ then $\fd{\bar x}{\bar y}$.
\item Augmentation: If $\fd{\bar x}{\bar y}$ then $\fd{\bar x,\bar z}{\bar y,\bar z}$.
\item Transitivity: If  $\fd{\bar x}{\bar y}$ and  $\fd{\bar y}{\bar z}$ then  $\fd{\bar x}{\bar z}$.
\end{itemize}
\end{prop}

A complete axiomatization of multivalued dependence is also possible as was shown by Beeri, Fagin and Howard:

\begin{prop}[\cite{Beeri:1977}]
Let $U$ be a finite set of variables, $D \cup \set{\varphi}$  a finite set of multivalued dependence atoms over the variables in $U$. Then $D \models \varphi$ iff $\varphi$ is derivable from $D$ with the following inference rules:
\begin{itemize}
\item Complementation: If $\bar x \cup \bar y \cup \bar z = U$, $\bar y \cap \bar z \subseteq \bar x$, and $\mvd{\bar x}{\bar y}$ then $\mvd{\bar x}{\bar z}$
\item Reflexivity: If $\bar y \subseteq \bar x$ then $\mvd{\bar x}{\bar y}$.
\item Augmentation: If $\mvd{\bar x}{\bar y}$ then $\mvd{\bar x,\bar z}{\bar y,\bar z}$.
\item Transitivity:  If  $\mvd{\bar x}{\bar y}$ and  $\mvd{\bar y}{\bar z}$ then  $\mvd{\bar x}{\bar z\setminus \bar y}$.\footnote{Here $\bar z \setminus \bar y$ is the set difference, i.e., the set of all variables in $\bar z$ not in $\bar y$.}
\end{itemize}
We are assuming that all $\bar x$, $\bar y$, and $\bar z$ are variables in $U$.
\end{prop}

However the embedded multivalued dependency is not axiomatizable as was shown by Sagiv and Walecka  in the following sense:

\begin{prop}[\cite{Sagiv:1982}]
There is no finite set of inference rules, where each inference rule is a recursive set of $k$-tuples of embedded multivalued dependencies, axiomatizing the consequence relation $D \models \varphi$ for embedded multivalued dependencies.
\end{prop}

This answers an open question stated in \cite{Gradel:2011}.

Galliani recently observed that $X \models \bar t \perp_{\bar s} \bar t'$ iff
$$X \models \exists \bar x \bar y \bar z \bigl( \bar x = \bar s \land \bar y = \bar t \land \bar z = \bar t' \land \forall \bar u \mvd{\bar x}{\bar y}\bigr),$$
where $\bar u$ is the domain of $X$. Thus, we get the following proposition.
\begin{prop}[\cite{Galliani:2011}]
The multivalued Dependence logic has the same strength as Independence logic, even at the level of formulas. 
\end{prop} 

Thus the definable sets of teams of both Independence logic and multivalued Dependence logic is exactly the sets of teams definable by existential second order sentences.

\section{Conclusion and discussion}\label{conclusion}

In this paper we have given truth conditions for monotone generalized 
quantifiers in logics using team semantics in such a way that the meaning of 
$L(Q)$-formulas remain the same when moving to team semantics, i.e., a team 
satisfies a formula of $L(Q)$ iff every assignment in the team satisfies the 
formula. It is also shown that the truth conditions in a natural way can be 
extended to deal with relations of dependence and independence between the 
quantifiers, in such a way that branching of two quantifiers $Q_1$ and $Q_2$ 
can be expressed by a linear quantifier prefix: $Q_1x Q_2 y \slash x$. 

We also gave truth conditions for non-monotone quantifiers by using an idea 
from Sher in \cite{Sher:1990} to add a largeness or maximality condition. For 
this condition the quantifier prefix $Q_1x Q_2 y \slash x$ comes close to the 
branching $\Br^S(Q_1,Q_2)xy$ defined in \cite{Sher:1990}, but they are not 
equivalent: In the prefix $Q_1x Q_2 y \slash x$ the second quantifier depends 
on the first in a weak sense, however in the case of $\Br^S(Q_1,Q_2)xy$ there 
is full symmetry in the sense that $\Br^S(Q_1,Q_2)xy$ is equivalent to 
$\Br^S(Q_2,Q_1) yx$.

Is there some way of treating the maximaility principle of Sher, 
$\Br^S(Q_1,Q_2)$, in a compositional way in the framework used in this paper?
Can other proposed principles, e.g. the one in \cite{Westerstahl:1987}, of 
branching in the non-monotone case be handled compositional in the same way?

The question of whether the notion of dependence and independence of 
(monotone) quantifiers can be handled on the atomic level is answered 
positively in the paper. However, the notion of dependence is not the 
functional dependence of Dependence logic, but rather a new kind of dependence 
atom, called multivalued dependence. This atom is not closed under taking 
subteams, but can be used to express branching of generalized quantifiers, 
which the functional dependence atom cannot: $\Br(Q_1,Q_2)xy$ is equivalent to 
$Q_1xQ_2y (\mvd{}{y} \land \ldots)$.

If a monotone quantifier $Q$ is definable in ESO, i.e., there is an ESO 
sentence $\sigma$ such that $M \models \sigma$ iff $M \in Q$, then it is easy 
to see that the strength on sentence level of the logic $\text{SBL}(Q)$ is 
just the strength of existential second order logic, ESO. This comes from the 
fact that $\text{SBL} \equiv \text{ESO}$ and by observing that for any formula 
$\varphi$ of $\text{SBL}(Q)$ we can find a sentence $\sigma$ of ESO such that 
$$M ,X \models   \varphi \text{ iff } (M,X(\bar x)) \models \sigma.$$ This is 
done by coding the truth conditions of $\varphi$ into the sentence $\sigma$.
Thus, $\SBL(Q) \equiv \SBL$ but the translation of $\SBL(Q)$ sentences into 
$\SBL$ is non-compositional.

For which quantifiers $Q$ are there compositional translations of $\SBL(Q)$ 
into SBL? In particular, is there a compositional translation of $\SBL(Q_0)$ 
into SBL, where $Q_0$ is the quantifier  ``there exists infinitely many''?

Of course we have not answered one of the basic questions regarding our 
definition of generalized quantifiers:
When introducing a monotone quantifier, which may not be definable in ESO, 
into Dependence logic, what is the strength of the resulting logic?

In connection with investigating the strength of these kinds of logics it 
might be worth mentioning Krynicki's result in \cite{Krynicki:1993} saying 
that there is a single quantifier $Q$ of type $\langle 4 \rangle$ such that 
every IF-logic sentence is equivalent to a sentence of $L(Q)$ over every 
structure with a pairing function. Is this also true for $\text{SBL}(Q)$, i.e., is 
there a single quantifier $Q'$ such that any sentence of $\text{SBL}(Q)$ is 
equivalent to a sentence of $L(Q')$ over any structure with a pairing 
function?

There is a connection of multivalued dependence with category theory through 
pullbacks, or fibered products. Proposition \ref{fagin77} gives a 
characterization of multivalued dependence in terms of natural join, which in 
turn has a characterization in terms of pullbacks:

In the category of teams, where the objects are teams and the morphisms are 
functions between teams, the natural join of $X$ and $Y$ is the pullback of 
$X$ and $Y$ over $Z=(X \restrictto \bar z) \cap (Y \restrictto \bar z)$, where 
$\bar z$ is $\dom(X) \cap \dom(Y)$. More precisely; let $\bar x$ and $\bar y$ 
be  $\dom(X)$ and $\dom(Y)$ respectively, then the following is a pullback 
diagram:
\begin{center}
\begin{tikzpicture}[node distance=4em, auto]
\node(P) {$X \bowtie Y$};
\node(X) [left of=P, below of=P]{$X$};
\node(Y) [right of=P, below of=P]{$Y$};
\node(Z) [left of=Y, below of=Y]{$Z$};
\draw[->] (P) to node [swap]{$\cdot\restrictto\bar x$} (X);
\draw[->] (P) to node {$\cdot\restrictto\bar y$} (Y);
\draw[->] (X) to node [swap]{$\cdot\restrictto\bar z$} (Z);
\draw[->] (Y) to node {$\cdot\restrictto\bar z$} (Z);
\end{tikzpicture}
\end{center}

Thus, by using Proposition \ref{fagin77}, we see that $X \models \mvd{\bar x}{\bar y}$ holds iff the commuting diagram
\begin{center}
\begin{tikzpicture}[node distance=4em, auto]
\node(P) {$X$};
\node(X) [left of=P, below of=P]{$X \restrictto \bar x \bar y$};
\node(Y) [right of=P, below of=P]{$X \restrictto \bar x \bar z$};
\node(Z) [left of=Y, below of=Y]{$X \restrictto \bar x$};
\draw[->] (P) to node {} (X);
\draw[->] (P) to node {} (Y);
\draw[->] (X) to node {} (Z);
\draw[->] (Y) to node {} (Z);
\end{tikzpicture}
\end{center}
where $\bar z$ is $\dom(X) \setminus \set{\bar x,\bar y}$, is a pullback. This suggests that there might be more, and deeper, connections between team semantics and category theory.

\end{document}